\documentclass[11pt]{amsart}

\usepackage[T1]{fontenc}
\usepackage{lmodern}
\usepackage{mathtools}
\usepackage{amssymb}
\usepackage{xcolor}
\usepackage{microtype}
\usepackage[a4paper,hmargin=30mm,vmargin=28mm]{geometry}
\usepackage[hidelinks]{hyperref}
\everymath{\displaystyle}
\numberwithin{equation}{section}
\newtheorem{theorem}{Theorem}[section]
\newtheorem{lemma}[theorem]{Lemma}
\newtheorem{proposition}[theorem]{Proposition}
\newtheorem{remark}[theorem]{Remark}
\newtheorem{corollary}[theorem]{Corollary}
\newtheorem{definition}[theorem]{Definition}
\allowdisplaybreaks[1]
\newcommand{\bR}{\mathbb{R}}

\newcommand{\bM}{\mathbb{M}}

\newcommand{\Mn}{\mathbb M_n}
\newcommand{\Mnp}{\mathbb M_n^+}
\newcommand{\Tr}{\operatorname{Tr}}
\newcommand{\diag}{\operatorname{diag}}
\newcommand{\spec}{\sigma}
\newcommand{\norm}[1]{\lVert #1\rVert}

\newcommand{\smod}[1]{\lvert #1\rvert_{\mathrm{sym}}}

\title[Solutions to some open problems]{Solutions to some open problems in matrix analysis}
\author{Mohamed Amine Aouichaoui}
\address{Department of Mathematics, Faculty of Sciences of Monastir, University of Monastir, 5019 Monastir, Tunisia}
\email{amine.aouichaoui@fsm.rnu.tn}

\author{Eun-Young Lee}
\address{Department of Mathematics, KNU-Center for Nonlinear Dynamics,
Kyungpook National University,
 Daegu 702-701, Korea}
\email{eylee89@knu.ac.kr}

\subjclass[2020]{Primary 15A42; Secondary 15A18, 15A60, 47A30}
\keywords{Unitary orbit, Horn inequalities, reciprocal power limit, positive block matrix, symmetric modulus}

\begin{document}

%\begin{abstract}
%This paper deals with several open questions in matrix analysis. We solve two questions of Audenaert and Kittaneh: (1) we improve the Aujla--Bourin subadditivity inequality, and (2) we show that a Hlawka-type inequality for the trace norm fails. We also address several questions posed by Bourin and the second author: (3) we improve the reciprocal Lie--Trotter theorem of Audenaert and Hiai; (4) we establish the sharpness of a series of inequalities, for instance, given three contractions $A,B,C$, we have the triangle inequality
%$$
%|A+B+C| \leq \frac{3}{4} I + |A|+|B|+|C|,
%$$
%where $3/4$ is optimal; and (5) we obtain an elegant eigenvalue inequality involving the symmetric modulus $(|X|+|X^*|)/2$, thereby contributing to the study of symmetric moduli initiated by Bourin, Lee, and  Zhang.
%\end{abstract}

\begin{abstract} We solve a problem posed by Audenaert and Kittaneh in 2012 by proving
that the Aujla--Bourin subadditivity inequality can be strengthened
to an optimal form. We also answer
questions posed by Bourin and the second author:
(1)
we significantly extend both the reciprocal Lie--Trotter theorem
of Audenaert and Hiai and Kato's limit theorem;

(2) we determine the best constants in inequalities for sums
of contractions and positive block matrices, including the
sharp triangle inequality
$$
|A+B+C| \leq \frac{3}{4} I + |A|+|B|+|C|
$$
for three contractions $A,B,C$; and
(3) 
we prove a remarkable eigenvalue inequality for the symmetric
modulus, thereby determining the best dimension-free lower bound
in a question studied by Bourin, Lee, and Zhang.

\end{abstract}

\maketitle

\section{Introduction}

Let $\Mn$ denote the algebra of complex $n\times n$ matrices,
and let $\Mnp$ be its positive semidefinite cone. For $X\in\Mn$,
we write
\[
 |X|=(X^*X)^{1/2},
 \qquad
 \smod{X}=\frac{|X|+|X^*|}{2}
\]
for its right modulus and symmetric modulus, respectively.
We use $\norm{X}$ for the operator norm and
$\|X\|_2=(\Tr(X^*X))^{1/2}$ for the Frobenius norm, where
$\Tr$ denotes the trace. A matrix $X$ is a contraction if
$\norm{X}\leq1$.

The eigenvalues of a Hermitian matrix $T\in\Mn$ are counted
with multiplicity and arranged in nonincreasing order:
\[
 \lambda_1(T)\geq\cdots\geq\lambda_n(T),
 \qquad
 \lambda(T)=(\lambda_1(T),\ldots,\lambda_n(T)).
\]
For Hermitian matrices $S,T$ of the same size, $S\leq T$
means that $T-S$ is positive semidefinite. There exists
a unitary matrix $U$ such that $S\leq UTU^*$ if and only if
$\lambda_j(S)\leq\lambda_j(T)$ for every $j$.
We write $I$ for the identity matrix, with a subscript when
its size needs to be specified.

For contractions $A_1,\ldots,A_k\in\Mn$, Bourin and Lee
proved that
\[
 \left|\sum_{j=1}^k A_j\right|
 \leq \frac{k}{4}I+\sum_{j=1}^k|A_j|.
\]
They showed that $k/4$ cannot be replaced by a smaller constant
when $k$ is even and conjectured the same conclusion for
odd $k\geq3$; see~\cite[Corollary~4.4 and Remark~4.5]{BourinLee2024}.
We prove their conjecture by constructing, for each odd
$k\geq3$, a family of $k$ Hermitian matrices in $\mathbb M_3$
with eigenvalues $-1,0,1$.

Bourin and Lee also asked about the eigenvalues of sums of
moduli when the sum of the matrices satisfies
$\left|\sum_{j=1}^mX_j\right|\geq I$;
see~\cite[Question~6.3]{BourinLee2026}.
For $X_1,\ldots,X_m\in\mathbb M_d$, with $m,d\geq2$ and
$d\in\{2n-1,2n\}$, we prove that
\[
 \left|\sum_{j=1}^mX_j\right|\geq I_d
 \quad\Longrightarrow\quad
 \lambda_n\left(\sum_{j=1}^m\smod{X_j}\right)\geq\frac12.
\]
For every $m\geq2$, the constant $1/2$ cannot be increased
when $(d,n)=(3,2)$ or $(d,n)=(6,3)$.
For ordinary moduli, we give, for every $0<\varepsilon<1$,
three matrices $X_1,X_2,X_3\in\mathbb M_3$ such that
\[
 X_1+X_2+X_3=I_3,
 \qquad
 \lambda_2\bigl(|X_1|+|X_2|+|X_3|\bigr)=\varepsilon.
\]

Audenaert and Kittaneh~\cite[Problem~5]{AK}
(Problem~6 in arXiv:1201.5232v3) asked whether
the Aujla--Bourin subadditivity inequality~\cite{AB} can be
strengthened to a double inequality. We prove that, for
$A,B\in\Mnp$ and a nonnegative concave function $f$ on
$[0,\infty)$, there exist unitaries $U,V\in\Mn$ such that
\[
 0\leq f(A+B)-Uf(A)U^*\leq Vf(B)V^*.
\]
The proof uses Fulton's theorem on majorized spectra and
a refinement of a Littlewood--Richardson lemma due to Buch,
recorded in~\cite[Lemma~3 and footnote~2]{Fulton}.

For a positive block matrix
\[
 \mathcal A=
 \begin{bmatrix}A&N\\N^*&B\end{bmatrix},
 \qquad A,B,N\in\Mn,
\]
Bourin and Lee~\cite[Corollary~2.4]{BourinLee2022} proved that
\[
 \lambda_1(\mathcal A)\leq\lambda_1(A+B)+r
\]
whenever $N$ is normal and its spectrum is contained in a disk
of radius $r$. We show that the coefficient of $r$ cannot be
decreased, even when the blocks have size $3$.
This answers~\cite[Question~2.5]{BourinLee2022}.

Bourin and Lee proved that the eigenvalues of
\[
 (A^pZ^*B^pZA^p)^{1/p}
\]
converge as $p\to\infty$ for $A,B\in\Mnp$ and $Z\in\Mn$.
They asked whether the matrices themselves converge;
see~\cite[Corollary~2.6 and the paragraph following it]{BourinLee2016}.

Using a unitary dilation technique, we significantly extend the
reciprocal Lie--Trotter theorem of Audenaert and Hiai~\cite{AudenaertHiai},
thereby obtaining convergence in norm.

The paper is organised as follows.
Section~\ref{sec:contractions} gives the examples for sums
of contractions and for ordinary moduli.
Section~\ref{sec:expansive} proves the lower bound for sums
of symmetric moduli.
Section~\ref{sec:double} contains the double inequality for
concave functions and its proof.
Section~\ref{sec:block} gives the examples for the
disk-radius coefficient.
Section~\ref{sec:limit} proves the convergence result.

\section{Sums of matrices and the right modulus}
\label{sec:contractions}

For  matrices $A_1,\ldots,A_k\in\Mn$, a sharp comparison  between the sum of the moduli and the modulus of the sum    states that
\begin{equation}\label{frobenius}
\| \, |A_1+\cdots + A_k|\, \|_2 \le \sqrt{\frac{1+\sqrt{k}}{2}} \, \left\| |A_1|+\cdots +|A_k|\right\|_2.
\end{equation}
This Frobenius norm inequality was proved in \cite{TaZ26} by Q.\,Tang and S.\,Zhang.  The case $k=2$ is due to J.\,Lin and Y.\,Zhang \cite{LZ22} and solves Lee's conjecture \cite{Lee10}.

If we want a comparison in the Loewner order instead of a norm comparison, it is then necessary to deal with bounded families of matrices, for instance with families of contractions (matrices with operator norm $\le1$).  For   contractions $A_1,\ldots,A_k\in\Mn$, Bourin and Lee proved that
\begin{equation}\label{eq:BL-contractions}
 \left|\sum_{j=1}^k A_j\right|
 \leq\frac{k}{4}I+\sum_{j=1}^k|A_j|.
\end{equation}
They also proved that $k/4$ is optimal for even $k$~\cite[Corollary~4.4]{BourinLee2024}. Their Remark~4.5 asks about odd $k$.

For even $k$, the case $k=2$ suffices. Bourin and
Lee~\cite[Corollary~4.4]{BourinLee2024} gave two contractions
$C_1,C_2\in\mathbb M_3$ for which the coefficient $1/2$ in
\[
 |C_1+C_2|\leq\frac12 I+|C_1|+|C_2|
\]
cannot be decreased. For $k=2n$, taking $n$ copies of $C_1$
and $n$ copies of $C_2$ shows that the coefficient $k/4$
in~\eqref{eq:BL-contractions} cannot be decreased either.
For odd $k\geq3$, we construct $k$ partial symmetries in
$\mathbb M_3$. We recall the definition.

\begin{definition}
A partial symmetry is a Hermitian matrix whose spectrum
is contained in $\{-1,0,1\}$.
\end{definition}

\begin{theorem}\label{thm:k-over-four}
For every integer $k>1$, the constant $k/4$ in~\eqref{eq:BL-contractions} is optimal. Partial symmetries in $\mathbb M_3$ suffice to attain it.
\end{theorem}

\begin{corollary}\label{corthm2.2}
For every integer $k>1$, the constant $k/4$ in~\eqref{eq:BL-contractions} is optimal for real contactions. Partial symmetries in $\mathbb M_6(\bR)$ suffice to attain it.
\end{corollary}

A natural question: can dimension $6$ in Corollary~\ref{corthm2.2} be reduced?
Proposition~\ref{cor:k-over-four} below shows that
dimension $4$ suffices. We first prove
Theorem~\ref{thm:k-over-four}.

\begin{proof}[Proof of Theorem \ref{thm:k-over-four}]
The even case is contained in~\cite[Corollary~4.4]{BourinLee2024}. Let $k\geq3$ be odd, let $e_0,e_1,e_2$ be the standard basis of $\mathbb C^3$, and put $\zeta=e^{2\pi i/k}$. For $j=0,\ldots,k-1$, set
\[
 v_j=\frac12e_1+\frac{\sqrt3}{2}\zeta^j e_2,
 \qquad
 A_j=\begin{bmatrix}0&v_j^*\\v_j&0_2\end{bmatrix}
\]
relative to $\mathbb C e_0\oplus\operatorname{span}\{e_1,e_2\}$. Each $A_j$ has eigenvalues $1,0,-1$ and
\[
 |A_j|=A_j^2=e_0e_0^*+v_jv_j^*.
\]
Since $\sum_{j=0}^{k-1}\zeta^j=0$, we have
\[
 \left|\sum_{j=0}^{k-1}A_j\right|
 =\diag\left(\frac k2,\frac k2,0\right),
 \qquad
 \sum_{j=0}^{k-1}|A_j|
 =\diag\left(k,\frac k4,\frac{3k}{4}\right).
\]
The largest eigenvalue of their difference is $k/4$. Thus, no smaller scalar can replace $k/4$ in~\eqref{eq:BL-contractions}.
\end{proof}

\begingroup

\begin{proposition}\label{cor:k-over-four}
For every integer $k\geq2$, the constant $k/4$
in~\eqref{eq:BL-contractions} cannot be decreased, even when
the matrices are real partial symmetries in
$\mathbb M_4(\mathbb R)$.
\end{proposition}

\begin{proof}
Fix $k\geq2$. For $j=0,\ldots,k-1$, put
\[
 \theta_j=\frac{2\pi j}{k},
 \qquad
 v_j=
 \begin{pmatrix}
  1/2\\
  (\sqrt3/2)\cos\theta_j\\
  (\sqrt3/2)\sin\theta_j
 \end{pmatrix},
 \qquad
 A_j=
 \begin{pmatrix}
  0&v_j^{\mathsf T}\\
  v_j&0_3
 \end{pmatrix}.
\]
Each $v_j$ is a unit vector. Hence $A_j$ is a real
partial symmetry and
\[
 |A_j|=A_j^2=
 \begin{pmatrix}
  1&0\\
  0&v_jv_j^{\mathsf T}
 \end{pmatrix}.
\]
Since
\[
 \sum_{j=0}^{k-1}\cos\theta_j
 =
 \sum_{j=0}^{k-1}\sin\theta_j
 =0,
\]
we have
\[
 \left|\sum_{j=0}^{k-1}A_j\right|
 =
 \diag\left(\frac{k}{2},\frac{k}{2},0,0\right).
\]
For the unit vector $x=(0,1,0,0)^{\mathsf T}$,
\[
 \left\langle
 \left|\sum_{j=0}^{k-1}A_j\right|x,x
 \right\rangle=\frac{k}{2},
 \qquad
 \left\langle
 \left(\sum_{j=0}^{k-1}|A_j|\right)x,x
 \right\rangle=\frac{k}{4}.
\]
Thus, if
\[
 \left|\sum_{j=0}^{k-1}A_j\right|
 \leq cI_4+\sum_{j=0}^{k-1}|A_j|,
\]
then $k/2\leq c+k/4$, so $c\geq k/4$.
\end{proof}
\endgroup

\begin{remark}\label{rem:colbrook}
After completing this manuscript, we learned of the resolution
of Problem MI-03 by Colbrook~\cite{ColbrookMI03}, dated
September 11, 2026. Colbrook proves that the constant $k/4$
in~\eqref{eq:BL-contractions} is optimal for every $k\geq2$;
in particular, this value is already attained by complex
rank-one contractions in $\mathbb M_2$.
Note that Colbrook's example does not cover
Proposition~\ref{cor:k-over-four}.
\end{remark}

The following proposition is related to a question of Bourin and Lee that we will  investigate in the next section.

\begin{proposition}\label{prop:ordinary-zero}
For every $0<\varepsilon<1$, there exist $X_1,X_2,X_3\in\mathbb M_3$ such that
\[
 X_1+X_2+X_3=I
\]
and
\begin{equation*}
 \lambda_2\left(|X_1|+|X_2|+|X_3|\right)=\varepsilon.
\end{equation*}
\end{proposition}

\begin{proof}
Put
\[
 a=\varepsilon^{-1}-\varepsilon,
 \qquad
 b=a+\sqrt{a^2+1},
 \qquad
 Q=\diag(b,\varepsilon,\varepsilon).
\]
The identity
\[
 b-b^{-1}=2a=2(\varepsilon^{-1}-\varepsilon)
\]
shows that $b+2\varepsilon=b^{-1}+2\varepsilon^{-1}$ so that
 $$\Tr Q=\Tr Q^{-1}.$$
Now, let $\omega=e^{2\pi i/3}$, and take the orthonormal basis
\[
 f_j=\frac1{\sqrt3}
       \begin{pmatrix}1\\\omega^j\\\omega^{2j}\end{pmatrix},
 \qquad j=0,1,2.
\]
Set
\[
 c_j=Q^{-1/2}f_j,
 \qquad d_j=Q^{1/2}f_j,
 \qquad X_{j+1}=c_jd_j^*.
\]
Each coordinate of $f_j$ has squared modulus $1/3$, so
\[
 \norm{c_j}^2=\frac{\Tr Q^{-1}}3
             =\frac{\Tr Q}3=\norm{d_j}^2.
\]
Hence $|X_{j+1}|=d_jd_j^*$. Summing over the orthonormal basis gives
\[
 \sum_{j=1}^3X_j
 =Q^{-1/2}\left(\sum_{j=0}^2f_jf_j^*\right)Q^{1/2}=I
\]
and
\[
 \sum_{j=1}^3|X_j|
 =Q^{1/2}\left(\sum_{j=0}^2f_jf_j^*\right)Q^{1/2}=Q.
\]
Since $b>1>\varepsilon$, the second largest eigenvalue of $Q$ is $\varepsilon$.
\end{proof}

\begin{remark} An example of two real matrices $X_1,X_2\in\mathbb M_3$ such that $X_1+X_2=I$ and
$$
\lambda_2\left(|X_1|+|X_2|\right)\simeq 0,8835
$$
is given by Zhang in \cite{TZhang2026}. Although $X_1$ and $X_2$ commute, their moduli need not commute.
\end{remark}

\begin{remark} Li~\cite{Li26} proved the following extension of~\eqref{frobenius}
for every nonnegative concave function $f$ on $[0,\infty)$:
\begin{equation}\label{frobeniusconcave}
\left\| f\left( |A_1+\cdots + A_k|\right)\right\|_2 \le \sqrt{\frac{1+\sqrt{k}}{2}} \, \left\| f(|A_1|)+\cdots +f(|A_k|)\right\|_2.
\end{equation}
In section 4, we will study some other nice matrix inequalities for concave functions.
\end{remark}

\section{Expansive decompositions and symmetric modulus}
\label{sec:expansive}

Bourin and Lee~\cite[Question~6.3]{BourinLee2026} asked for
lower bounds on
\[
 \lambda_n\left(\sum_{j=1}^m\smod{X_j}\right)
 \quad\text{and}\quad
 \lambda_n\left(\sum_{j=1}^m|X_j|\right)
\]
for matrices $X_1,\ldots,X_m\in\mathbb M_d$ satisfying
\[
 \left|\sum_{j=1}^mX_j\right|\geq I_d,
 \qquad d\in\{2n-1,2n\}.
\]

Proposition~\ref{prop:ordinary-zero} shows that
$\lambda_2(|X_1|+|X_2|+|X_3|)$ can be arbitrarily small
for three matrices in $\mathbb M_3$, even when
$X_1+X_2+X_3=I_3$.
For symmetric moduli, the following theorem gives
\[
 \lambda_n\left(\sum_{j=1}^m\smod{X_j}\right)\geq\frac12
\]
whenever $m,d\geq2$,
$d\in\{2n-1,2n\}$, and
$\left|\sum_{j=1}^mX_j\right|\geq I_d$.

\begin{theorem}\label{thm:sym-lower}
Let $X_1,\ldots,X_m\in\mathbb M_d$, where $m,d\geq2$
and $d\in\{2n-1,2n\}$. If
$\left|\sum_{j=1}^mX_j\right|\geq I_d$, then
\begin{equation}\label{eq:sym-lower}
 \lambda_n\left(\sum_{j=1}^m\smod{X_j}\right)\geq\frac12.
\end{equation}
For   any $m\ge2$, the constant $1/2$ is optimal, both for $(d,n)=(3,2)$ and  for $(d,n)=(6,3)$.
\end{theorem}

\begin{proof}
Put $X=\sum_{j=1}^mX_j$ and
\[
 P=\sum_{j=1}^m|X_j^*|,
 \qquad
 Q=\sum_{j=1}^m|X_j|.
\]
The standard block positivity relation as in \cite[Sec.\ 3]{BourinLee2026} gives
\[
 \begin{bmatrix}|X_j^*|&X_j\\X_j^*&|X_j|\end{bmatrix}\geq0
\]
and, summing over $j$, 
\[
 \begin{bmatrix}P&X\\X^*&Q\end{bmatrix}\geq0.
\]
Write $X=U|X|$, where $U$ is unitary. Conjugation by $\diag(U^*,I_d)$ gives
\[
 \begin{bmatrix}P_0&|X|\\|X|&Q\end{bmatrix}\geq0,
 \qquad P_0=U^*PU.
\]
Evaluating on $(h,-h)$ yields $P_0+Q\geq2|X|\geq2I_d$.

Weyl's eigenvalue inequality~\cite[Chapter~III]{Bhatia} now gives
\begin{align*}
 2
 &\leq\lambda_d(P_0+Q)\\
 &\leq\lambda_n(P_0)+\lambda_{d-n+1}(Q)\\
 &\leq\lambda_n(P)+\lambda_n(Q)\\
 &\leq2\lambda_n(P+Q).
\end{align*}
The second line uses 
$$
d=1+(d-1), \qquad d-1= (d-n)+(n-1)
$$
where the assumption $d\ge 2$ ensures that both $d-n\ge 0$ and $n-1\ge 0$.
The third line uses $d-n+1\geq n$. Since $\sum_j\smod{X_j}=(P+Q)/2$, this proves~\eqref{eq:sym-lower}. 

 Proposition~\ref{prop:sym-sharp} below proves optimality of the constant $1/2$ for $(d,n)=(3,2)$ and  $m=2$ (and so by   adding some zero summands for any $m\ge 2$). The comment after the proof of Proposition~\ref{prop:sym-sharp} establishes the optimality of the constant $1/2$ for $(d,n)=(6,3)$. 
\end{proof}

\begin{proposition}\label{prop:sym-sharp}
For every $0<\varepsilon<1$, there exist $X_1,X_2\in\mathbb M_3$ such that $X_1+X_2$ is unitary and
\begin{equation}\label{eq:sym-example}
 \lambda_2\bigl(\smod{X_1}+\smod{X_2}\bigr)
 =\frac{1+\varepsilon}{2}.
\end{equation}
\end{proposition}

\begin{proof}
Let
\[
 D=\diag(1,\varepsilon,\varepsilon^{-1}),
 \qquad
 U=\begin{bmatrix}0&1&0\\0&0&1\\1&0&0\end{bmatrix},
 \qquad
 C=UD^{1/2},
 \qquad E=D^{-1/2}.
\]
Choose the orthonormal basis
\[
 u_1=e_1,
 \qquad
 u_2=\frac{e_2+e_3}{\sqrt2},
 \qquad
 u_3=\frac{e_2-e_3}{\sqrt2},
\]
and put $c_j=Cu_j$, $d_j=Eu_j$, and $Y_j=c_jd_j^*$. The norms satisfy
\[
 \norm{c_1}=\norm{d_1}=1,
 \qquad
 \norm{c_j}^2=\norm{d_j}^2
 =\frac{\varepsilon+\varepsilon^{-1}}2
 \quad(j=2,3).
\]
The rank-one modulus formula therefore gives
\[
 |Y_j|=d_jd_j^*,
 \qquad
 |Y_j^*|=c_jc_j^*.
\]
Since $\sum_ju_ju_j^*=I_3$,
\begin{equation}\label{eq:rank-one-sums}
 \sum_{j=1}^3Y_j=CE^*=U,
 \qquad
 \sum_{j=1}^3|Y_j|=D^{-1},
 \qquad
 \sum_{j=1}^3|Y_j^*|=UDU^*.
\end{equation}
Moreover, $c_1\perp c_2$ and $d_1\perp d_2$. So, the initial spaces of $Y_1,Y_2$ are orthogonal, as are their final spaces. It follows that
\[
 |Y_1+Y_2|=|Y_1|+|Y_2|,
 \qquad
 |Y_1^*+Y_2^*|=|Y_1^*|+|Y_2^*|.
\]
Set $X_1=Y_1+Y_2$ and $X_2=Y_3$. By~\eqref{eq:rank-one-sums}, their sum is $U$ and
\begin{align*}
 \smod{X_1}+\smod{X_2}
 &=\frac{UDU^*+D^{-1}}2\\
 &=\diag\left(\frac{1+\varepsilon}{2},
               \varepsilon^{-1},
               \frac{1+\varepsilon}{2}\right).
\end{align*}
This proves~\eqref{eq:sym-example}.
\end{proof}

For the two matrices $X_1,X_2$ in Proposition~\ref{prop:sym-sharp}, whose sum is unitary, take
$X_1\oplus X_1$ and $X_2\oplus X_2$ in $\mathbb M_6$. The third largest eigenvalue of the sum of their symmetric moduli is again $(1+\varepsilon)/2$. This proves the optimality of $1/2$ in Theorem \ref{thm:sym-lower} for $(d,n)=(6,3)$.

\section{Strong form of Aujla--Bourin's theorem}
\label{sec:double}

Let $A,B\in\Mnp$ and let $f:[0,\infty)\to[0,\infty)$ be concave.  Aujla and Bourin~\cite[Theorem~2.1]{AB} proved that
$$
f(A+B) \le Uf(A)U^*+Vf(B)V^*
$$
for some unitaries $U$ and $V$. Taking traces gives
Rotfel'd's inequality~\cite{Rot69}:
\begin{equation}\label{Rot}
\Tr f(A+B) \le \Tr f(A) + \Tr f(B).
\end{equation}
This trace inequality also holds for real-valued concave
functions $f$ on $[0,\infty)$ with $f(0)\geq0$.
The unitary-orbit inequality need not hold under these weaker
assumptions; see Zhang~\cite{TZhang26b}.

We use the following result of Fulton~\cite{Fulton} to prove
a stronger form of the Aujla--Bourin inequality.
Let $Z,A_1,\ldots,A_m\in\Mn$ be Hermitian.
There exist unitaries $U_1,\ldots,U_m\in\Mn$ such that
\[
 Z\leq\sum_{t=1}^m U_tA_tU_t^*
\]
if and only if
\[
 \sum_{k\in K}\lambda_k(Z)
 \leq
 \sum_{t=1}^m\sum_{i\in I_t}\lambda_i(A_t)
\]
for every Horn tuple $(I_1,\ldots,I_m;K)$ of subsets of
$\{1,\ldots,n\}$ with common cardinality $1\leq r\leq n$.
Here the tuple with
$I_1=\cdots=I_m=K=\{1,\ldots,n\}$ is included; it gives
\[
 \Tr Z\leq\sum_{t=1}^m\Tr A_t.
\]
For the definition of the Horn tuples, see~\cite[Section~2]{Fulton}.
We use this convention throughout the section.

We first prove the following result.

\begin{lemma}\label{lem:completion}
Let \(F,X,Y\in\Mnp\). Suppose that
\[
    \lambda_i(X)\leq\lambda_i(F),
    \qquad 1\leq i\leq n,
\]
and that
\[
    F\leq U_0 X U_0^*+V_0 Y V_0^*
\]
for some unitaries \(U_0,V_0\in\Mn\).
Then, there exist unitaries \(U,V\in\Mn\) such that
\begin{equation}\label{eq:completion}
    0\leq F-UXU^*\leq VYV^*.
\end{equation}
\end{lemma}

\begin{proof}

Put
\[
    \alpha=\lambda(X),\qquad
    \beta=\lambda(Y),\qquad
    \gamma=\lambda(F).
\]

We first prove that there exist positive semidefinite matrices
\(P,D\in\Mnp\) such that
\begin{equation}\label{eq:spectral-completion-revised}
    \lambda(P)=\alpha,\qquad
    \lambda_i(D)\leq\beta_i\quad(1\leq i\leq n),
    \qquad
    \lambda(P+D)=\gamma.
\end{equation}

We use Fulton's spectral characterization of the relation
\(C\leq A+B\); see~\cite[Theorem~1]{Fulton}.
By hypothesis,
\[
    F\leq U_0 X U_0^*+V_0 Y V_0^*.
\]
Hence the eigenvalue lists
\[
    \alpha=\lambda(X),\qquad
    \beta=\lambda(Y),\qquad
    \gamma=\lambda(F)
\]
form an admissible spectral triple for the relation \(C\leq A+B\).
By Fulton's spectral characterization, this is equivalent to saying
that \((\alpha,\beta,\gamma)\) satisfies all the majorized Horn
inequalities.

Moreover, by assumption,
\begin{equation}\label{eq:alpha-gamma}
    \alpha_i
    =\lambda_i(X)
    =\lambda_i(U_0 X U_0^*)
    \leq \lambda_i(F)
    =\gamma_i,
    \qquad 1\leq i\leq n.
\end{equation}

We first consider the case in which
\(\alpha,\beta,\gamma\) have rational entries. Multiplying all three
vectors by a common positive integer, and rescaling at the end, it is
enough to treat the case in which \(\alpha,\beta,\gamma\) are
partitions with at most \(n\) parts.

Write \(\mu\subseteq\nu\) if
\[
    \mu_i\leq\nu_i
    \qquad\text{for every }i,
\]
and denote by \(c_{\mu,\nu}^{\rho}\) the corresponding
Littlewood--Richardson coefficient.

Fulton's partition formulation of the majorized Horn theorem
\cite[Theorem~5 and the following remark]{Fulton} gives a partition
\(\widehat{\gamma}\) such that
\[
    \gamma\subseteq\widehat{\gamma}
    \qquad\text{and}\qquad
    c_{\alpha,\beta}^{\widehat{\gamma}}>0.
\]

Since $\alpha\subseteq\gamma\subseteq\widehat{\gamma}$,
Buch's refinement of Lemma~3, recorded in
\cite[footnote~2]{Fulton}, gives a partition
$\delta\subseteq\beta$ such that
\begin{equation}\label{eq:LR-refinement}
 c_{\alpha,\delta}^{\gamma}>0.
\end{equation}

The Horn--Littlewood--Richardson theorem now gives positive
semidefinite matrices \(P,D\in\Mnp\) such that
\[
    \lambda(P)=\alpha,\qquad
    \lambda(D)=\delta,\qquad
    \lambda(P+D)=\gamma.
\]

Since \(\delta\subseteq\beta\), we have
\[
    \lambda_i(D)=\delta_i\leq\beta_i,
    \qquad 1\leq i\leq n.
\]
Thus \eqref{eq:spectral-completion-revised} holds in the integral case,
and hence, after clearing denominators and rescaling, in the rational
case.

We now pass to arbitrary real spectra. Let \(\mathcal C\) denote the
set of triples \((a,b,c)\in\mathbb{R}^{3n}\) satisfying the majorized
Horn inequalities of Fulton, together with
\[
    a_1\geq\cdots\geq a_n\geq0,\qquad
    b_1\geq\cdots\geq b_n\geq0,\qquad
    c_1\geq\cdots\geq c_n\geq0,
\]
and
\[
    a_i\leq c_i,
    \qquad 1\leq i\leq n.
\]
All these inequalities have rational coefficients. Hence
\(\mathcal C\) is a rational polyhedral cone. 
By the preceding rational case, every rational point
\[
    (a,b,c)\in\mathcal C\cap\mathbb{Q}^{3n}
\]
is realized by positive semidefinite matrices \(P,D\) satisfying
\[
    \lambda(P)=a,\qquad
    \lambda_i(D)\leq b_i,\qquad
    \lambda(P+D)=c.
\]
Since \((\alpha,\beta,\gamma)\in\mathcal C\), and since rational
points are dense in the rational polyhedral cone \(\mathcal C\), there
exists a sequence of rational points
\[
    (\alpha^{(m)},\beta^{(m)},\gamma^{(m)})
    \in\mathcal C\cap\mathbb{Q}^{3n}
\]
such that
\[
    \alpha^{(m)}\longrightarrow\alpha,\qquad
    \beta^{(m)}\longrightarrow\beta,\qquad
    \gamma^{(m)}\longrightarrow\gamma.
\]
For each \(m\), choose positive semidefinite matrices
\(P_m,D_m\) such that
\[
    \lambda(P_m)=\alpha^{(m)},\qquad
    \lambda_i(D_m)\leq\beta_i^{(m)},\qquad
    \lambda(P_m+D_m)=\gamma^{(m)}.
\]
The sequences \((P_m)\) and \((D_m)\) are bounded. Indeed,
\[
    \|P_m\|
    =\lambda_1(P_m)
    =\alpha_1^{(m)}
\]
and
\[
    \|D_m\|
    =\lambda_1(D_m)
    \leq\beta_1^{(m)},
\]
and the right-hand sides remain bounded.
Since $\Mn$ is finite-dimensional and both sequences are bounded,
we may pass to a common subsequence such that
\[
    P_m\longrightarrow P,\qquad
    D_m\longrightarrow D
\]
for some \(P,D\in\Mnp\). 
Since the  eigenvalues depend continuously on the matrix, we
obtain
\[
    \lambda(P)
    =\lim_{m\to\infty}\lambda(P_m)
    =\alpha,
\]
and, for every \(i\),
\begin{equation}\label{eq:eigenvalue-continuity}
    \lambda_i(D)
    =\lim_{m\to\infty}\lambda_i(D_m)
    \leq
    \lim_{m\to\infty}\beta_i^{(m)}
    =\beta_i.
\end{equation}
Finally,
\[
    P_m+D_m\longrightarrow P+D,
\]
and therefore
\[
    \lambda(P+D)
    =
    \lim_{m\to\infty}\lambda(P_m+D_m)
    =
    \gamma.
\]
Thus \eqref{eq:spectral-completion-revised} holds for arbitrary real
spectra.

We now return to the original matrices \(F,X,Y\). Since
\(\lambda(P)=\lambda(X)\), there exists a unitary \(U\in\Mn\) such that
\[
    P=UXU^*.
\]
Since \(\lambda(P+D)=\lambda(F)\), there exists a unitary
\(W\in\Mn\) such that
\[
    F=W(P+D)W^*.
\]
Conjugating the matrices \(P,D\) by \(W\), and replacing \(U\) by
\(WU\), we may therefore assume, without loss of generality, that
\[
    F=P+D=UXU^*+D.
\]
In particular,
\[
    F-UXU^*=D\geq0.
\]
It remains to compare \(D\) with \(Y\). By \eqref{eq:eigenvalue-continuity},
\[
    \lambda_i(D)\leq\lambda_i(Y),
    \qquad 1\leq i\leq n.
\]
Equivalently, there exists a unitary \(V\in\Mn\) such that
\(
    D\leq VYV^*.
\)
Consequently,
\[
    0\leq F-UXU^*\leq VYV^*,
\]
which is \eqref{eq:completion}.
\end{proof}

\begin{theorem}\label{thm:double}

Let \(A,B\in\Mnp\) and let
\(f:[0,\infty)\to[0,\infty)\) be concave. Then, there exist unitaries
\(U,V\in\Mn\) such that
\begin{equation}\label{eq:double}
    0\leq f(A+B)-Uf(A)U^*
    \leq Vf(B)V^*.
\end{equation}
\end{theorem}

\begin{proof}
Set
\[
    F=f(A+B),\qquad
    X=f(A),\qquad
    Y=f(B).
\]
These matrices are positive semidefinite.

Since \(A\leq A+B\), we have
\[
    \lambda_i(A)\leq\lambda_i(A+B),
    \qquad 1\leq i\leq n.
\]
As a nonnegative concave function on \([0,\infty)\) is
nondecreasing, we infer
\[
    \lambda_i(X)
    =f(\lambda_i(A))
    \leq f(\lambda_i(A+B))
    =\lambda_i(F),
    \qquad 1\leq i\leq n.
\]
By the theorem of Aujla and Bourin
there exist unitaries
\(U_0,V_0\in\Mn\) such that
\[
    F\leq U_0 X U_0^*+V_0 Y V_0^*,
\]
and Lemma~\ref{lem:completion} now gives
\eqref{eq:double}.
\end{proof}

\begin{remark}
\label{rem:functional-horn}
By Fulton's spectral characterization, the Aujla--Bourin
inequality is equivalent to
\[
\sum_{k\in K} f\bigl(\lambda_k(A+B)\bigr)
\leq
\sum_{i\in I} f\bigl(\lambda_i(A)\bigr)
+
\sum_{j\in J} f\bigl(\lambda_j(B)\bigr)
\]
for every Horn tuple $(I,J;K)$.
Applying a concave function to the usual Horn inequalities does
not by itself give these inequalities. For results on concave
functions and Horn inequalities, see Zhou~\cite{JZhou26}.
\end{remark}

Theorem~\ref{thm:double} cannot in general be strengthened by requiring
$Vf(B)V^*\leq f(A+B)$ for the same unitaries $U,V$.
The following example shows this already in dimension $2$.

\begin{proposition}\label{prop:triple-counterexample}
Let
\[
 f(t)=\frac{t}{1+t}\quad(t\geq0),
 \qquad
 A=\begin{pmatrix}1&0\\0&0\end{pmatrix},
 \qquad
 B=\frac12\begin{pmatrix}1&1\\1&1\end{pmatrix}.
\]
There are no unitaries $U,V\in\mathbb M_2(\mathbb C)$ for which all three
inequalities
\begin{equation}\label{eq:triple-comparison}
 \begin{split}
 f(A+B)&\leq Uf(A)U^*+Vf(B)V^*,\\
 Uf(A)U^*&\leq f(A+B),\\
 Vf(B)V^*&\leq f(A+B)
 \end{split}
\end{equation}
hold. The matrices $A,B$ are real rank-one orthogonal projections,
and $f$ is nonnegative, smooth, strictly increasing and strictly
concave on $[0,\infty)$, with $f(0)=0$.
\end{proposition}

\begin{proof}
We have $A^2=A=A^*$, $B^2=B=B^*$ and
$\Tr A=\Tr B=1$. Also, the function $f(t)=1-\frac{1}{1+t}$ is obvioulsy strictly increasing and
strictly concave on $[0,\infty)$; it is also operator monotone
and operator concave.
Since $f(0)=0$ and $f(1)=1/2$,
\[
 f(A)=\frac12A,
 \qquad
 f(B)=\frac12B.
\]
Using $f(T)=I-(I+T)^{-1}$ for $T\geq0$, we obtain
\[
 F:=f(A+B)
 =I-(I+A+B)^{-1}
 =\frac17\begin{pmatrix}4&1\\1&2\end{pmatrix}>0,
 \qquad \Tr F=\frac67.
\]
Suppose that $U,V$ satisfy~\eqref{eq:triple-comparison}, and set
\[
 X=Uf(A)U^*,\qquad Y=Vf(B)V^*,
\]
and
\[
 R=F^{-1/2}XF^{-1/2},\qquad S=F^{-1/2}YF^{-1/2}.
\]
Both $R$ and $S$ have rank one. The three inequalities imply
\[
 0\leq R\leq I,\qquad 0\leq S\leq I,\qquad R+S\geq I.
\]
A positive semidefinite rank-one matrix bounded above by $I$ has
trace at most one. Consequently,
\[
 2\leq\Tr(R+S)=\Tr R+\Tr S\leq2.
\]
Thus $R+S-I\geq0$ has trace zero, so $R+S=I$.
Multiplying on both sides by $F^{1/2}$ gives $X+Y=F$.
Taking traces now yields the contradiction
\[
 1=\Tr f(A)+\Tr f(B)=\Tr(X+Y)=\Tr F=\frac67.
\]
\end{proof}

\section{The disk-radius coefficient for positive block matrices}
\label{sec:block}

Let
\[
\mathcal B=\begin{bmatrix}A&X\\X^*&B\end{bmatrix}\geq0,
\qquad A,B,X\in\Mn,
\]
and let $f:[0,\infty)\to[0,\infty)$ be concave.
For every Schatten $p$-norm, $1\leq p\leq\infty$, we have
\[
\|f(\mathcal B)\|_p\leq\|f(A)\|_p+\|f(B)\|_p;
\]
see~\cite[Corollary~3.5]{BL12}.
For $p=1$ and $X=A^{1/2}B^{1/2}$, this implies
Rotfel'd's inequality~\eqref{Rot}. 

\begingroup

In another direction,
\endgroup
Bourin and Lee proved the following estimate~\cite[Corollary~2.4]{BourinLee2022}. If
\[
 \mathcal A=\begin{bmatrix}A&N\\N^*&B\end{bmatrix}\geq0,
 \qquad A,B,N\in\Mn,
\]
$N$ is normal, and $\spec(N)$ lies in a disk of radius $r$, then
\begin{equation}\label{eq:block-bound}
 \lambda_{1+2j}(\mathcal A)
 \leq\lambda_{1+j}(A+B)+r,
 \qquad j=0,\ldots,n-1.
\end{equation}
Their Question~2.5 asks whether the coefficient of $r$ can be reduced and whether it is already best possible for $j=0$.

\begin{theorem}\label{thm:radius-sharp}
For every $r>0$ and $\varepsilon>0$, there are $A,B\in\mathbb M_3^+$ and a normal $N\in\mathbb M_3$ such that
\[
 \begin{bmatrix}A&N\\N^*&B\end{bmatrix}\geq0,
 \qquad
 \spec(N)\subseteq\{z\in\mathbb C:|z|\leq r\},
\]
and
\begin{equation}\label{eq:radius-sharp}
 \lambda_1\!\left(\begin{bmatrix}A&N\\N^*&B\end{bmatrix}\right)
 >\lambda_1(A+B)+r-\varepsilon.
\end{equation}
Hence the coefficient of $r$ in~\eqref{eq:block-bound} is best possible already for $j=0$.
\end{theorem}

\begin{proof}
Let $T>2$ and put
\[
 U=\begin{bmatrix}0&1&0\\0&0&1\\1&0&0\end{bmatrix},
 \qquad
 A_T=\diag\left(T,\frac2T,\frac T2\right),
 \qquad
 C_T=\diag\left(T,\frac T2,\frac2T\right).
\]
Define
\[
 B_T=U^*C_TU,
 \qquad
 \mathcal A_T=\begin{bmatrix}A_T&U\\U^*&B_T\end{bmatrix}.
\]
Then,
\[
 \diag(I_3,U)\,\mathcal A_T\,\diag(I_3,U^*)
 =\begin{bmatrix}A_T&I_3\\I_3&C_T\end{bmatrix}.
\]
After a permutation of the coordinates, the matrix on the right is the direct sum of
\[
 \begin{bmatrix}T&1\\1&T\end{bmatrix},
 \qquad
 \begin{bmatrix}2/T&1\\1&T/2\end{bmatrix},
 \qquad
 \begin{bmatrix}T/2&1\\1&2/T\end{bmatrix}.
\]
All three blocks are positive semidefinite. Their eigenvalues give
\[
 \spec(\mathcal A_T)
 =\left\{T+1,\ T-1,\ \frac T2+\frac2T,\ \frac T2+\frac2T,\ 0,\ 0\right\},
\]
so $\lambda_1(\mathcal A_T)=T+1$. On the other hand,
\[
 A_T+B_T=\diag\left(T+\frac2T,T+\frac2T,T\right).
\]
It follows that
\begin{equation}\label{eq:radius-gap}
 \lambda_1(\mathcal A_T)-\lambda_1(A_T+B_T)=1-\frac2T.
\end{equation}
Take $A=rA_T$, $B=rB_T$, and $N=rU$. Since $U$ is unitary, $N$ is normal and its spectrum lies in the closed disk of radius $r$ centered at zero. The block matrix is $r\mathcal A_T\geq0$. By~\eqref{eq:radius-gap}, the difference between the largest eigenvalues of $r\mathcal A_T$ and $A+B$ equals $r(1-2/T)$. Any
\[
 T>\max\{2,2r/\varepsilon\}
\]
therefore gives~\eqref{eq:radius-sharp}.
\end{proof}

\begin{remark}\label{rem:two-by-two-blocks}
For a normal off-diagonal block of size $2$, the term $r$
in~\eqref{eq:block-bound} can be omitted.
Indeed, the numerical range of a normal matrix $N\in\mathbb M_2$
is the line segment joining its two eigenvalues, possibly a single point.
Its elliptical width is therefore zero. By
\cite[Theorem~2.1 and the discussion following Corollary~3.1]{BourinLee2022},
if
\[
 \mathcal A=\begin{bmatrix}A&N\\N^*&B\end{bmatrix}\geq0,
 \qquad A,B,N\in\mathbb M_2,
\]
with $N$ normal, then
\[
 \lambda_{1+2j}(\mathcal A)\leq\lambda_{1+j}(A+B),
 \qquad j=0,1.
\]
For scalar blocks, $\lambda_1(\mathcal A)\leq\Tr\mathcal A=A+B$.
Hence blocks of size at least $3$ are needed for the positive gap
in Theorem~\ref{thm:radius-sharp}.
\end{remark}

\section{\texorpdfstring{Anti Lie--Trotter and Kato power limit theorems}{Anti Lie--Trotter and Kato power limit theorems}}
\label{sec:limit}

In this last section, we use
 \begingroup
block matrices, more precisely
\endgroup
 a unitary dilation to derive a three-matrix version of the
reciprocal Lie--Trotter theorem of Audenaert and Hiai. Their theorem states that
\begin{equation}\label{eq:AH}
 \lim_{p\to\infty}\bigl(R^{p/2}S^pR^{p/2}\bigr)^{1/p}
\end{equation}
exists in norm for positive semidefinite matrices $R,S$ of the same size~\cite[Theorem~2.5]{AudenaertHiai}. The following corollary answers the question posed
after~\cite[Corollary~2.6]{BourinLee2016}.

\begin{proposition}\label{prop:matrix-limit}
For $A,B\in\Mnp$ and $Z\in\Mn$, the limit
\begin{equation}\label{eq:power-limit}
 \lim_{p\to\infty}\bigl(A^pZ^*B^pZA^p\bigr)^{1/p}
\end{equation}
exists in norm.
\end{proposition}

\begin{proof}
Put $c=\max\{1,\norm{Z}\}$ and $K=Z/c$. Complete the columns of the isometry
\[
 \begin{bmatrix}K\\(I_n-K^*K)^{1/2}\end{bmatrix}
\]
to a unitary $W\in\mathbb M_{2n}$. Its upper-left $n\times n$ block is $K$. Set
\[
 R=A^2\oplus0_n,
 \qquad
 S=W^*(B\oplus0_n)W.
\]
These matrices are positive semidefinite, and for every $p>0$,
\[
 R^{p/2}S^pR^{p/2}
 =\bigl(A^pK^*B^pKA^p\bigr)\oplus0_n.
\]
Taking positive $p$th roots and applying~\eqref{eq:AH}, we obtain the norm convergence of $(A^pK^*B^pKA^p)^{1/p}$. Finally,
\[
 \bigl(A^pZ^*B^pZA^p\bigr)^{1/p}
 =c^{2/p}\bigl(A^pK^*B^pKA^p\bigr)^{1/p},
\]
~and $c^{2/p}\to1$.
\end{proof}

\begingroup

We may state a more general form of the reciporocal Lie-Trotter formuale of Audenaert-Hiai. This is the mutivariate version in our last theorem.

\begin{theorem}\label{thm:matrix-limit}
For $A,B_1,\ldots,B_m\in\Mnp$ and $Z_1,\ldots,Z_m\in\Mn$, the limit
\begin{equation}\label{eq:sum-power-limit}
 \lim_{p\to\infty}\left(A^p\left\{\sum_{k=1}^mZ_k^*B_k^pZ_k\right\}A^p\right)^{1/p}
\end{equation}
exists in norm.
\end{theorem}

\vskip 5pt
\begin{proof} Following the proof of \cite[Corollary 2.7]{BourinLee2016} we apply Proposition~\ref{prop:matrix-limit} to the matrices
$\tilde{A},\tilde{B},\tilde{Z}\in \bM_{mn}$,
\begin{equation*}\label{eqblock}
\tilde{A}=\begin{bmatrix} A& 0_n&\cdots &0_n \\
0_n &0_n&\cdots &0_n \\
\vdots &\vdots &\ddots &\vdots \\
0_n& 0_n&\cdots &0_n \\
\end{bmatrix}, \
\tilde{B}=\begin{bmatrix} B_1& 0_n&\cdots &0_n \\
0_n &B_2&\cdots &0_n \\
\vdots &\vdots &\ddots &\vdots \\
0_n& 0_n&\cdots &B_m \\
\end{bmatrix}, \
\tilde{Z}=\begin{bmatrix} Z_1& 0_n&\cdots &0_n \\
Z_2 &0_n&\cdots &0_n \\
\vdots &\vdots &\ddots &\vdots \\
Z_m& 0_n&\cdots &0_n \\
\end{bmatrix}
\end{equation*}
where $0_n$ stands for the zero matrix in $\bM_n$.
Convergence of the block matrices in norm entails the convergence
of the upper-left block and completes the proof.
\end{proof}
\endgroup

\begin{remark}\label{rem:kato-limit}
Taking $A=Z_1=\cdots=Z_m=I$ in
Theorem~\ref{thm:matrix-limit} gives the norm convergence of
$(B_1^p+\cdots+B_m^p)^{1/p}$ for $B_1,\ldots,B_m\in\Mnp$.
Kato's limit theorem~\cite{Kato1979} identifies this limit as
\begin{equation}\label{eq:kato-limit}
 \lim_{p\to\infty}(B_1^p+\cdots+B_m^p)^{1/p}
 =B_1\vee\cdots\vee B_m,
\end{equation}
where the right-hand side denotes the supremum with respect
to the spectral order (Olson order).
\end{remark}

\section*{Acknowledgements}
We are very grateful to the referee for a careful reading of
the manuscript and for detailed comments and suggestions
that improved the paper and its presentation.

\section*{Declarations}

\noindent\textbf{Competing interests.}
The authors declare that they have no competing interests.

\medskip

\noindent\textbf{Funding.}
This research received no external funding.

\medskip

\noindent\textbf{Data availability statement.}
Data sharing is not applicable to this article, as no datasets were generated or analyzed during the current study.

\end{document}